%% file: main.tex
\documentclass[11pt]{article}
\input{CL_marco}

\title{Solving Partial Differential Equations with Random Feature Models}
\author{Chunyang Liao\footnote{Email: \url{liaochunyang@math.ucla.edu}} --- University of California, Los Angeles}
\date{\today}

\begin{document}
\pagenumbering{arabic}
\maketitle

\begin{abstract}
Machine learning based partial differential equations (PDEs) solvers have received great attention in recent years. 
Most progress in this area has been driven by deep neural networks such as physics-informed neural networks (PINNs) and kernel method.
In this paper, we introduce a random feature based framework toward efficiently solving PDEs. 
Random feature method was originally proposed to approximate large-scale kernel machines and can be viewed as a shallow neural network as well.
We provide an error analysis for our proposed method along with comprehensive numerical results on several PDE benchmarks. 
In contrast to the state-of-the-art solvers that face challenges with a large number of collocation points, our proposed method reduces the computational complexity.
Moreover, the implementation of our method is simple and does not require additional computational resources.
Due to the theoretical guarantee and advantages in computation, our approach is proven to be efficient for solving PDEs.
\end{abstract}


\section{Introduction}
Solving partial differential equations (PDEs) is a fundamental question in science and engineering. 
Traditional numerical methods include finite element method and finite difference method. 
Recently, the use of machine learning tools for solving PDEs, or in general any complex scientific tasks, has led to a new area of {\it scientific machine learning}.
Unlike traditional numerical methods, machine learning (ML) based methods do not rely on complex mesh designs and intricate numerical techniques. 
Therefore, it enables simpler, faster, and more convenient implementation and use. 
The most prominent ML-based solver is physics-informed neural network (PINN) \cite{Raissi2019PINN}, which uses a deep neural network to approximate the PDE solution. 
Given a set of collocation points in a spatiotemporal domain $\Omega$, we parametrize the PDE solution as a neural network satisfying PDE, boundary conditions, and initial conditions at given collocation points.
This approach leads to solving an optimization problem where the objective function measures the PDE residual with respect to some loss functional. 
Finding the solution of PDE is equivalent to optimizing the neural network parameters by using variants of stochastic gradient descent.  
PINN and its variations have achieved great success in learning the PDE solutions. However, it is extremely hard and expensive to optimize all parameters in the deep neural network. To reduce the computational complexity, some recent works \cite{WANG2024Extreme} proposed to use randomized neural networks to solve PDEs. Randomized neural network is a special type of neural networks with some parameters are randomly generated from a known probability distribution, instead of being optimized. This strategy was proven to reduce the computational time as well as maintain the approximation accuracy, see numerical experiments in \cite{WANG2024Extreme}. 
While these neural network based methods are often used in practice, the theoretical analysis relies on the universal approximation property of deep neural network, which shows the existence of a network of a requisite size achieving a certain error rate. 
However, the existence result does not guarantee that the network is computable in practice.
 
Instead of using deep neural networks, kernel method/Gaussian process (GPs) are also used to learn the PDE solutions \cite{xu2024kernel, fang2024solving, CHEN20211kernelPDE, BATLLE2025Error}. 
The main idea of such method is to approximate the solution of a given PDE as an element in a reproducing kernel Hilbert space (RKHS). 
This element will be found by solving an optimal recovery problem constrained by a PDE at collocation points \cite{CHEN20211kernelPDE,Foucart2022LearningFN}.
An optimal recovery problem can also be interpreted as maximum a posterior (MAP) estimation for a Gaussian process constrained by a PDE.
The key to solving an optimal recovery problem is a celebrated representer theorem that characterizes the minimizer as a finite-dimensional representation formula, which is easy to implement and interpret. 
Moreover, kernel-based PDE solving methods are also supported by rigorous theoretical foundation. Specifically, the authors provided a detailed priori error estimates in \cite{BATLLE2025Error}. 
However, the computational efficiency of kernel method can be a significant drawback. Specifically, it does not scale well when the sample size is large. For example, given $m$ training collocation points, kernel method requires $\Oc(m^3)$ training time and $\Oc(m^2)$ to store the kernel matrix, which is often computationally infeasible when $m$ is large.

To overcome the computational bottleneck in kernel method, random feature method was proposed to approximate large-scale kernel machines \cite{Rahimi2007RFM}. 
The main idea of random features is mapping data into a low-dimensional randomized feature space. Then, the kernel matrix is approximated by a low-rank matrix, which reduces the computational and storage costs of operating on kernel matrix.
The random feature model can be viewed as a randomized two-layer neural network.
The weights connecting input layer and single hidden layer are randomly generated from a known distribution rather than trainable parameters. Only the weights on the output layer are trainable.

In this paper, we propose a random feature based PDE solver. Since random feature model is a type of randomized neural network and an approximation of kernel method, the computational complexity can be reduced significantly. Moreover, we provide a convergence analysis of our method. 
The key contributions of our work are summarized as follows:
\begin{itemize}
\item {\bf Framework:} We propose a random feature framework for solving PDEs. By minimizing the PDE residuals, our method does not require the construction of kernel matrix. In practice, this framework allows us to use the modern automatic differential libraries, which is straightforward and convenient.

\item {\bf Convergence Analysis:} We provide a detailed convergence analysis of our proposed framework under some mild and widely used assumptions on PDEs. Our convergence analysis contains two steps: the first step follows the standard convergence analysis of kernel method \cite{BATLLE2025Error}; the second step concerns the approximation of kernel method using random feature method. To the best of our knowledge, this is the first work providing a convergence analysis of random feature PDE solving method. 

\item {\bf Numerical Experiments:} We test the performance of our framework with nonlinear elliptic PDEs, (high-dimensional) nonlinear Poisson PDEs, Allen-Cahn equation, and Advection-diffusion equation. Our method requires less computational resources to train the model with a similar or better performance compared with the existing methods on all benchmarks. We also numerically verify the convergence rate obtained from our analysis for all problems.
\end{itemize}

The remaining of this paper is organized as follows. In Section \ref{Sec:RF}, we give an overview of random feature method and provide a framework for solving PDEs using random feature model along with error analysis. 
Section \ref{Sec:numerics} is dedicated to the numerical experiments for solving PDEs with random feature method. We compare our method with PINN on several benchmarks and provide a empirical convergence study. 
We conclude the paper with a summary of the results and some possible future directions in Section \ref{Sec:conclusion}.

\section{Random Feature}
\label{Sec:RF}

\subsection{Overview of Random Feature Method}

To introduce random feature method, we first give a short introduction to kernel method, which is also known as kernel trick. 
It is one of the popular techniques for capturing nonlinear relations between features and targets. Let $\xb,\xb'\in X\subset\R^d$ be samples and $\phi:X\to\Hc$ be a feature map transforming samples to a high-dimensional (even infinite-dimensional) reproducing kernel Hilbert space $\Hc$ where the mapped data can be learned by a linear model.
In practice, the explicit expression of feature map $\phi$ is not necessarily known to us. The inner produce between $\phi(\xb)$ and $\phi(\xb')$ endowed by $\Hc$ can be computed by using a kernel function $k(\cdot,\cdot):X\times X\to\R$, i.e.
\begin{equation*}
k(\xb,\xb') = \langle \phi(\xb), \phi(\xb')\rangle_\Hc.
\end{equation*}
Due to the ease of computing the inner product, kernel method is effective for nonlinear learning problems with a wide range of successful applications \cite{smola2001kernel}.
However, kernel method does not scale well to extremely large datasets. For example, given $m$ training samples, kernel regression requires $\Oc(m^3)$ training time and $\Oc(m^2)$ to store the kernel matrix, which is often computationally infeasible when $m$ is large.
Random feature method is one of the most popular techniques to overcome the computational challenges of kernel method. 
The theoretical foundation of random (Fourier) feature builds on the following classical result from harmonic analysis.
\begin{theorem}[Bochner \cite{bochner1955harmonic}]
A continuous shift-invariant kernel $k(\xb,\xb') = k(\xb-\xb')$ on $\R^d$ is positive definite if and only if $k(\delta)$ is the Fourier transform of a non-negative measure.
\end{theorem}
Bochner’s theorem guarantees that the Fourier transform of kernel $k(\delta)$ is a proper probability distribution if the kernel is scaled properly. 
Precisely, given a shift-invariant kernel function $k(\delta)$ on $\R^d$, the probability distribution $\rho(\omegab)$ is derived by applying the Fourier transform, i.e.
\begin{equation}
\rho(\omegab) = \int_{\R^d} \exp(-i\langle \omegab,\delta\rangle) k(\delta) d\delta.
\end{equation}
Applying Bochner's Theorem, we are able to represent the kernel function as an integral form.
Denote the probability distribution by $\rho(\omegab)$, we have 
\begin{equation}
\label{kernel}
k(\xb,\xb') = \int_{\R^d} \exp(i\langle \omegab,\xb-\xb'\rangle)d\rho(\omegab) = \int_{\R^d} \exp(i\langle \omegab,\xb\rangle) \overline{\exp(i\langle \omegab,\xb'\rangle)}d\rho(\omegab),   
\end{equation}
where the first inequality holds because of the inverse Fourier transform.
Using Monte Carlo sampling technique, we randomly generate $N$ i.i.d samples $\{\omegab_k\}_{k\in[N]}$ from $\rho(\omegab)$ and define a {\it random Fourier feature map} $\phi:\R^d\to\C^N$ as \footnote{It depends on the i.i.d samples $\{\omegab_k\}_{k\in[N]}$ from $\rho(\omegab)$. To simplify the notation, we omit the dependency on $\{\omegab_k\}_{k\in[N]}$ when we define $\phi$.}
\begin{equation}
\phi(\xb) = \frac{1}{\sqrt{N}}\Big[\exp(i\langle \omegab_1, \xb\rangle), \dots, \exp(i\langle\omegab_N,\xb\rangle) \Big]^T \in\C^N.
\end{equation}
Using random Fourier feature map, we can define a kernel function $\hat{k}(\xb,\xb'):X\times X\to\R$ as
\begin{equation*}
\hat{k}(\xb,\xb') := \langle \phi(\xb),\phi(\xb')\rangle. 
\end{equation*}
Kernel function $\hat{k}$ is finite dimensional since the random Fourier feature map $\phi$ transforms data to a finite dimensional space $\C^N$. We can also use {\it random cosine features} to approximate any shift-invariant kernel. Specifically, setting random feature $\omegab\in\R^d$ generated from $\rho(\omegab)$ and $b\in\R$ sampled from uniform distribution on $[-\pi,\pi]$, the random cosine feature is defined as 
$\cos(\langle \omegab,\xb\rangle + b)$. Similarly, we can define a {\it random cosine feature map} $\phi:\R^d\to\R^N$ as 
\begin{equation}
\phi(\xb) = \frac{1}{\sqrt{N}}\Big[\cos(\langle \omegab_1, \xb\rangle + b_1), \dots, \cos(\langle\omegab_N,\xb\rangle+b_N) \Big]^T \in\R^N,
\end{equation}
using random i.i.d samples $\{\omegab_k,b_k\}_{k\in[N]}$, and hence define a kernel function $\hat{k}:X\times X\to\R$ using random cosine feature map.
We could approximate the original shift-invariant kernel function $k$ defined in \eqref{kernel} by the finite-dimensional kernel $\hat{k}$.
Utilizing random features allows efficient learning with $\Oc(mN^2)$ time and $\Oc(mN)$ storage capacity.

\subsection{Random Feature Regression}
Now, we set ourselves to the regression setting where our aim is to learn a function $f:\R^d\to\R$ using training samples $\{(\xb_j, y_j)\}_{j\in[m]}$. 
To implement the random feature model, we first draw $N$ i.i.d random features $\{\omegab_k\}_{k\in[N]} \subset \R^d$ from a probability distribution $\rho(\omegab)$, and then construct an approximation for target function $f$ taking the form 
\begin{equation}
\label{rf_train}
f^\sharp(\xb) = \sum_{k=1}^N c_k^\sharp \phi( \xb, \omegab_k).
\end{equation}
We assume that the $m$ sampling points $\xb_j$’s are drawn from a certain distribution with the corresponding output values
\begin{equation*}
y_j = f(\xb_j) + e_j, \quad \mbox{ for all } j\in[m],
\end{equation*}
where $e_j$ is the measurement noise. 
Let $\Ab\in\R^{m\times N}$ be the random feature matrix defined component-wise by $\Ab_{j,k}=\phi( \xb_j, \omegab_k)$ for $j\in[m]$ and $k\in[N]$. Training the random features model \eqref{rf_train} is equivalent to finding the coefficient vector $\cb^\sharp\in\R^N$ such that $\Ab\cb^\sharp \approx \yb$, where $\cb^\sharp = [c_1^\sharp,\dots,c_N^\sharp]^\top\in\R^N$ and $\yb=[y_1,\dots,y_m]\in\R^m$.

In the under-parameterized regime where we have more measurements than features ($m\geq N$), the coefficients are trained by solving the (regularized) least squares problem:
\begin{equation*}
\cb_\lambda^\sharp \in \argmin_{\cb\in\R^N} \|\Ab\cb-\yb\|_2^2 + \lambda \|\cb\|_2^2,
\end{equation*}
where $\lambda>0$ is the regularization parameter. It is also referred to ridge regression since the ridge regularization term $\lambda\|\cb\|_2^2$ is added. 

Recently, over-parametrized models have received great attention since those trained models not only fit the training samples exactly but also predict well on unseen test data \cite{Belkin2018kernel,Liang2020Kernel}. In the over-parametrized regime, we have more features than measurements ($N\geq m$), and we consider training the coefficient vector $\cb^\sharp\in\R^N$ using the min-norm interpolation problem:
\begin{equation*}
\cb^\sharp \in \argmin_{\cb\in\R^N} \|\cb\|_2^2 \quad \mbox{ subject to } \Ab\cb = \yb.
\end{equation*}
This problem is also referred to ridgeless regression problem since the solution $\cb^\sharp$ can be viewed as the limit of $\cb_\lambda^\sharp$ as $\lambda\to0$.

The generalization analysis of random features models have been of recent interest  \cite{Rahimi2008weighted,Rudi2017generalizationRFM,E2020NN,Sun2018SVM_RF,Mei2019TheGE,chen2024conditioning,Liao2024Differ}. In \cite{Rahimi2008weighted}, the authors showed that the random feature model yields a test error of $\Oc(N^{-\frac{1}{2}}+m^{-\frac{1}{2}})$ when trained on Lipschitz loss functions. Therefore, the generalization error is $\Oc(N^{-\frac{1}{2}})$ for large $N$ if $m\asymp N$. However, the model is trained by solving a constrained optimization problem which is rarely used in practice. 
In \cite{Rudi2017generalizationRFM}, it was shown that for $f$ in an RKHS, using $N=\Oc(\sqrt{m}\log(m))$ features is sufficient to achieve a test error of $\Oc(m^{-\frac{1}{2}})$ with squared loss. 
In \cite{E2020NN}, the authors showed that a regularized model can achieve $N^{-1}+m^{-\frac{1}{2}}$ risk provided that the target function belonging to an RKHS. Nevertheless, results in \cite{Rudi2017generalizationRFM,E2020NN} depend on the assumptions of kernel and a certain decay rate of second moment operator, which may be difficult to verify in practice. 
Extending the results of random feature models from squared loss to 0-1 loss, the authors of \cite{Sun2018SVM_RF} showed that the support vector machine with random features $N\ll m$ can achieve the learning rate faster than $\Oc(m^{-\frac{1}{2}})$ on a training set with $m$ samples.
In \cite{Mei2019TheGE}, the authors computed the precise asymptotic bound of the test error, in the limit $N,m,d \to\infty$ with $N/d$ and $m/d$ fixed. 
In \cite{chen2024conditioning}, the authors derived non-asymptotic bounds including both the under-parametrized setting using (regularized) least square problem and the over-parametrized setting using min-norm problem or sparse regression. Their results relied on the condition numbers of the random feature matrix and indicated double descent behavior in random feature models. However, the target function space is a subset of a RKHS, which limits the approximation ability of random feature model. 
In \cite{Liao2024Differ}, the authors consider a RKHS as target function space and derived a similar non-asymptotic bound by utilizing different proof techniques.

\begin{table}[!htbp]
\centering
\begin{tabular}{c|c|c}
Kernel name  & $k(\xb,\xb')$ & $\rho(\omegab)$ \\\hline
Gaussian kernel & $\exp(-\gamma\|\xb-\xb'\|_2^2)$, $\gamma>0$  & $\left(2\pi(2\gamma)^2\right)^{-d/2}\exp(-\frac{\|\omegab\|_2^2}{2(2\gamma)^2})$ \\
Laplace Kernel & $\exp(-\gamma\|\xb-\xb'\|_1)$, $\gamma>0$ & $\left(\frac{2}{\pi}\right)^d\prod\limits_{j=1}^d\frac{\gamma}{\gamma^2+\omega_j^2}$ \\
\end{tabular}
\caption{Commonly used kernels and the corresponding Fourier density}
\label{tab:kernel}
\end{table}

Furthermore, the random feature model can be viewed as a two-layer (one hidden layer) neural network where the weights connecting input layer and hidden layer and biases are sampled randomly and independently from a known distribution. Only the weights of the output layer are trainable using training samples, and hence it leads to solving a convex optimization problem when training random feature models.
There are other types of randomized neural network, including the random vector functional link (RVFL) network \cite{Deedell2024RVFL,MALIK2023RVFL}, and the extreme learning machine (ELM) \cite{HUANG2006ELM,Huang2006Universal,Wang2022ELM}, among others. 
Sharing the same structure as the random feature model, RVFL network is also a shallow neural network where the input-to-hidden weights and biases are randomly selected. 
However, the motivations of two models are different. RVFL networks were designed to address the difficulties associated with training deep neural networks and weights are usually sampled from uniform distribution. 
Random feature models were originally used to approximate large-scale kernel machines, and hence random weights depend on the kernel function, see Table \ref{tab:kernel} for some examples of commonly used kernels and the corresponding densities. 
Extreme learning machine is one type of deep neural network (more than two hidden layers) in which all the hidden-layer weights are randomly selected and then fixed. Only the output-layer coefficients are trained.
Theoretical guarantees on RVFL and ELM suggest that they are universal approximators, see \cite{Igelnik1995randomized}. 

\subsection{Random Feature Models for Solving PDEs}

In this section, we present our framework for solving PDEs using random feature models. Let us consider the PDE problem of the general form
\begin{equation}
\label{PDE}
\begin{aligned}
& \Pc[u](\xb) = 0, && \xb\in\Omega \\
& \Bc[u](\xb) = 0, && \xb \in \partial\Omega, 
\end{aligned}
\end{equation}
where $\Omega\subset\R^d$ is the domain with the boundary $\partial\Omega$, $\Pc$ is the interior differential operator and $\Bc$ is the boundary differential operator. 
For the sake of brevity, we assume that the PDE is well-defined pointwise and has a unique strong solution throughout this paper.
We propose to solve the PDE problem \eqref{PDE} by using a random feature model. More precisely, let $\{\xb_j\}_{j\in[M]}$ be a collection of collocation points such that $\{\xb_j\}_{j\in[M_\Omega]}$ is a collection of points in the interior of $\Omega$ and $\{\xb_j\}_{j=M_\Omega+1}^M$ is a set of points on the boundary $\partial\Omega$. 
Random features $\{\omegab_k\}_{k\in[N]}$ are randomly drawn from a known distribution $\rho(\omegab)$. The random feature model takes the form
\begin{equation}
\label{RF_form}
u^\sharp(\xb) = \sum_{k=1}^N c_k^\sharp \phi( \xb, \omegab_k)
\end{equation}
We consider the over-parametrization regime where the number of random features $N$ is greater than the number of collocation points $M$. This assumption comes from the fact that collocation points are hard to obtain in practice.

We train the random feature model by solving the following optimization problem:
\begin{equation}
\label{MinnormRF_PDE}
\begin{aligned}
& \minimize_{\cb\in\R^N} && \|\cb\|_2^2 \\
& \mbox{s.t.} && \Pc[u^\sharp](\xb_j) = 0, \quad \mbox{ for } j=1,\dots,M_\Omega \\
& && \Bc[u^\sharp](\xb_j) = 0, \quad \mbox{ for } j=M_\Omega+1,\dots,M
\end{aligned}
\end{equation}
We wish to find an approximation of the true solution $u$ with the min-norm random feature model satisfying the PDE and boundary data at collocation points. 

{\color{red}

\begin{algorithm}
\caption{Random Feature Model based PDE solver}
\label{Alg:rf1}
\begin{algorithmic}
\State\textbf{Inputs:} Collocation points $\{\xb_j\}_{j\in[M]}$, number of random features $N$, and probability distribution $\rho$
\State\textbf{Outputs:} random feature approximation $u^\sharp$
\State 1. Sample $N$ random features $\{ \omegab_k \}_{k\in[N]}$ from $\rho$ independently.
\State 2. Consider the random feature model taking the form \eqref{RF_form}.
\State 3. Solve the optimization problem \eqref{reg_RF_PDE} to produce the coefficient vector $\cb^\sharp$.
\State 4. Return the approximated PDE solution $u^\sharp$ with coefficient vector $\cb^\sharp$ obtained from Step 3.
\end{algorithmic}
\end{algorithm}
}

\begin{example}[Linear PDE]
If the PDE is linear, we can rewrite the constraints in \eqref{MinnormRF_PDE} as a linear system. Consider the following linear PDE 
\begin{equation*}
\begin{aligned}
-\Delta u(\xb) + u(\xb) &= f(\xb), \quad && \xb\in\Omega \\
u(\xb) &= g(\xb), \quad && \xb\in\partial\Omega, 
\end{aligned}
\end{equation*}
we can write the condition as the following linear system
\begin{equation}
\label{linear_PDE}
\begin{bmatrix}
\Ab \\\hline \Bb 
\end{bmatrix}\cb = \begin{bmatrix}
\fb \\\hline \gb
\end{bmatrix},
\end{equation}
where $\Ab\in\R^{M_\Omega\times N}$ and $\Bb\in\R^{(M-M_\Omega)\times N}$ are defined component-wise by $\Ab_{j,k} = \phi( \xb_j, \omegab_k) - \Delta\phi( \xb_j, \omegab_k)$ \footnote{Precisely, $\Delta\phi( \xb_j, \omegab_k)$ means that evaluation of $\Delta\phi( \xb, \omegab_k)$ at point $\xb=\xb_j$. } and by $\Bb_{j,k} = \phi( \xb_j, \omegab_k)$, respectively. Two vectors on the right-hand side are defined as $\fb=[f(\xb_1),\dots,f(\xb_{M_\Omega})]^\top\in\R^{M_\Omega}$ and $\gb=[g(\xb_{M_\Omega+1}),\dots,g(\xb_M)]^\top\in\R^{M-M_\Omega}$. 
In this case we compute $\cb\in\R^N$ by using the least squares method if the system is overdetermined or using the min-norm method if the system is underdetermined. However, this example cannot be generalized to nonlinear PDEs since the nonlinear PDE problem results in a nonlinear term in terms of coefficient vector $\cb$.
\end{example}
For general nonlinear PDE problems, we can neither write optimization problem \eqref{MinnormRF_PDE} as a solving linear system problem nor solve it directly. Therefore, we may solve an unconstrained optimization problem with regularization instead,
\begin{equation}
\label{reg_RF_PDE}
\minimize_{\cb\in\R^N} \|\cb\|_2^2 + \lambda_1 \sum_{j=1}^{M_\Omega}\left(\Pc[u^\sharp](\xb_j)\right)^2 + \lambda_2 \sum_{j=M_\Omega+1}^M \left(\Bc[u^\sharp](\xb_j)\right)^2,
\end{equation}
where $\lambda_1,\lambda_2>0$ are regularization parameters. When $\lambda_1,\lambda_2\to0$, the solution of \eqref{reg_RF_PDE} converges to the solution of \eqref{MinnormRF_PDE}. 
In practice, we use variants of stochastic gradient descent and the modern automatic differential libraries to solve problem \eqref{reg_RF_PDE}.
In scenarios where PDEs are challenging, a large number of collocation points are required to capture the solution details. 
Compared with the framework using the standard kernel matrix to construct the solution approximation in \cite{xu2024kernel}, our proposed random feature method approximates the kernel matrix, and hence reduces the computational cost and accelerate computation involving the standard kernel matrix (and its inverse) when dealing with massive collocation points.

\subsection{Convergence Analysis}
In this section, we show the convergence analysis of our method. 
Our convergence analysis relies on the standard convergence analysis of kernel method in \cite{BATLLE2025Error} and the kernel approximation by using random features.
Recall the kernel-based method for solving PDEs in \cite{CHEN20211kernelPDE,BATLLE2025Error}, a reproducing kernel Hilbert space (RKHS) $\Hc$ is chosen and we aim to solve the following
\begin{equation}
\label{Minnorm_KernelPDE}
\begin{aligned}
& \minimize_{u\in\Hc} && \|u\|_\Hc \\
& \mbox{s.t.} && \Pc[u^\sharp](\xb_j) = 0, \quad \mbox{ for } j=1,\dots,M_\Omega \\
& && \Bc[u^\sharp](\xb_j) = 0, \quad \mbox{ for } j=M_\Omega+1,\dots,M
\end{aligned}
\end{equation}
We first state the main assumptions on the domain $\Omega$ and its boundary $\partial\Omega$, the PDE operators $\Pc$ and $\Bc$, and the reproducing kernel Hilbert space $\Hc$.
\begin{assumption}
\label{ass}
The following assumptions hold:
\begin{itemize}
\item {\bf (C1) Regularity of the domain and its boundary} $\Omega\subset\R^d$ with $d>1$ is a compact set and $\partial\Omega$ is a smooth connected Riemannian manifold of dimension $d-1$ endowed with a geodesic distance $\rho_{\partial\Omega}$.
\item {\bf (C2) Stability of the PDE} There exist $\gamma>0$ and $k,t\in\N$ satisfying $d/2<k+\gamma$ and $(d-1)/2<t+\gamma$, and $s,\ell\in\R$ so that for any $r>0$ it holds that, for any $u_1,u_2\in B_r(H^\ell(\Omega))$,
\begin{equation*}
\|u_1-u_2\|_{H^\ell(\Omega)} \leq C\left( \|\Pc(u_1) - \Pc(u_2)\|_{H^k(\Omega)} + \|\Bc(u_1) - \Bc(u_2)\|_{H^t(\partial\Omega)} \right),
\end{equation*}
and for any $u_1, u_2\in B_r(H^s(\Omega))$,
\begin{equation*}
\|\Pc(u_1) - \Pc(u_2)\|_{H^{k+\gamma}(\Omega)} + \|\Bc(u_1) - \Bc(u_2)\|_{H^{t+\gamma}(\partial\Omega)} \leq C\|u_1-u_2\|_{H^s(\Omega)},
\end{equation*}
where $C= C(r)>0$ is a constant independent of $u_1$ and $u_2$.
\item {\bf (C3)} The RKHS $\Hc$ is continuously embedded in $H^s(\Omega)$.
\end{itemize}
\end{assumption}

Item (C1) is a standard assumption when analyzing PDEs. Item (C2) assumes that the PDE to be Lipschitz well-posed with respect to the right hand side/source term. It relates to the analysis of nonlinear PDEs and is independent of our numerical scheme. 
Assumption (C3) dictates the choice of the RKHS $\Hc$ , and in turn the kernel, which should be carefully selected based on the regularity of the strong solution $u$. 
We are now ready to state the first theorem which concerns the convergence rate of kernel method

\begin{theorem}[Theorem 3.8, \cite{BATLLE2025Error}] 
\label{Thm:kernel}
Suppose Assumption \ref{ass} is satisfied and denote the unique strong solution of \eqref{PDE} by $u\in\Hc$. Let $\hat{u}$ be a minimizer of \eqref{Minnorm_KernelPDE} with interior collocation points $X_\Omega$ and collocation points on the boundary $X_{\partial\Omega}$. Define the fill-in distances
\begin{equation*}
h_\Omega := \sup_{\xb'\in\Omega}\inf_{\xb\in X_\Omega} \|\xb-\xb'\|_2, \qquad  h_{\partial\Omega} := \sup_{\xb'\in\partial\Omega}\inf_{\xb\in X_{\partial\Omega}} \rho_{\partial\Omega}(\xb,\xb'),
\end{equation*}
and set $h=\max(h_\Omega, h_{\partial\Omega})$. Then there exists a constant $h_0$ so that if $h<h_0$ then
\begin{equation*}
\|u-\hat{u}\|_{H^s(\Omega)} \leq Ch^\gamma\|u\|_\Hc,
\end{equation*}
where $C>0$ is a constant independent of $h$ and $u$.
\end{theorem}

\begin{remark}
The above theorem relies on the assumption that the unique strong solution $u$ belongs to a RKHS $\Hc$. This is a standard assumption in the theoretical analysis. Sobolev space $H^s(\Omega)$, which is an appropriate function space for studying PDEs, is indeed a reproducing kernel Hilbert space provided that $s>d/2$ where $d$ is the dimension of the domain, i.e $\Omega\subset\R^d$. To relax this assumption, we could assume that the true solution $u$ can be approximated by an element $\Tilde{u}\in\Hc$. We plan to leave this as a future direction.
\end{remark}

With the kernel minimizer $\hat{u}\in\Hc$ at hand, our next step is approximating $\hat{u}$ using random features. We first adopt an alternative representation of preselected RKHS $\Hc$. 
Denote the corresponding random Fourier feature map (or random cosine feature map) by $\phi:X\to\C^N(\R^N)$, then we define the following function space 
\begin{equation}
\Fc(\rho) := \left\{f(\xb)=\int_{\R^d}\alpha(\omegab)\phi(\xb,\omegab)d\rho(\omegab): \|f\|^2_\rho = \E_{\omegab} [\alpha(\omegab)^2] <\infty \right\},
\end{equation}
where $\rho(\cdot)$ is the Fourier transform density associated with kernel $k$. 
Notice that the completion of $\Fc(\rho)$ is a Hilbert space equipped with RKHS norm $\|f\|_\rho$. 
Recall the Proposition 4.1 in \cite{Rahimi2008RFM}, it is indeed the reproducing kernel Hilbert space $\Hc$ with associated kernel function $k$. The endowed norms $\|\cdot\|_\Hc$ and $\|\cdot\|_\rho$ are equivalent. 

In the next theorem, we address the approximation ability of finite sum random feature model taking the form \eqref{RF_form}. 
\begin{theorem}
\label{Thm:RF}
Let $f$ be a function from $\Fc(\rho)$. Suppose that the random feature map $\phi$ satisfies $|\phi(\xb,\omegab)|\leq1$ for all $\xb\in X$ and $\omegab\in\R^d$. Then for any $\delta\in(0,1)$, there exists $c^\sharp_1,\dots,c^\sharp_N$ so that the function
\begin{equation}
\label{RF_approximator}
f^\sharp(\xb) = \sum_{k=1}^N c^\sharp_k\phi(\xb,\omegab_k) 
\end{equation}
satisfies
\begin{equation*}
\left| f(\xb) - f^\sharp(\xb)\right| \leq \frac{12\|f\|_\rho\log(2/\delta)}{\sqrt{N}}     
\end{equation*}
with probability at least $1-\delta$ over $\omegab_1,\dots,\omegab_N$ drawn i.i.d from $\rho(\omegab)$.
\end{theorem}
\begin{proof}
We first introduce notations $\alpha_{\leq T}(\omegab) = \alpha(\omegab)\indicator_{\left|\alpha(\omegab)\right|\leq T}$ and $\alpha_{>T} = \alpha(\omegab) - \alpha_{\leq T}(\omegab)$ for any $T>0$. Then we define 
\begin{equation}
c_k^\sharp = \alpha_{\leq T}(\omegab_k) \quad \mbox{ for all } k\in[N],
\end{equation}
where $\omegab_k$'s are i.i.d samples following a probability distribution with density $\rho(\omegab)$, and hence define $f^\sharp(\xb)$ in \eqref{RF_approximator} using $c^\sharp_k$'s.
We can show that
\begin{equation*}
\E f^\sharp(\xb) = \E_{\omegab} \left[ \alpha_{\leq T}(\omegab)\phi(\xb, \omegab) \right].
\end{equation*}
By utilizing the triangle inequality, we decompose the error into two terms
\begin{equation}
\left| f(\xb) - f^\sharp(\xb)\right| \leq \underbrace{\left| f(\xb) - \E f^\sharp(\xb)\right|}_{I_1} + \underbrace{\left| \E f^\sharp(\xb) - f^\sharp(\xb)\right|}_{I_2}.
\end{equation}
We first bound term $I_1$. Recalling the definitions of $f$ and $\alpha_{\leq T}(\omegab)$, we bound term $I_1$ as
\begin{equation}
\begin{aligned}
\label{term1}
\left| f(\xb) - \E f^\sharp(\xb) \right|^2 =& \Big| \E_{\omegab} \left[ \alpha_{>T}(\omegab)\phi(\xb,\omegab) \right] \Big|^2 
\leq \E_{\omegab} \left[\alpha(\omegab)\right]^2 \E_{\omegab} \big[ \indicator_{\left|\alpha(\omegab)\right|> T}\phi(\xb,\omegab)\big]^2  \\
=& \E_{\omegab} \left[\alpha(\omegab)\right]^2 \Pbb\big(\alpha(\omegab)^2>T^2\big)
\leq \frac{\Big(\E_{\omegab}[\alpha(\omegab)^2]\Big)^2}{T^2}
= \frac{\|f\|_\rho^4}{T^2}
\end{aligned}
\end{equation}
where we use the Cauchy-Schwarz inequality in the first line and the Markov's inequality in the second line. 

Next, we bound term $I_2$. For any $\xb\in X$, we define random variable $Z(\omegab) = \alpha_{\leq T}(\omegab)\phi(\xb,\omegab)$ and let $Z_1, \dots, Z_N$ be $N$ i.i.d copies of $Z$ defined by $Z_k = Z(\omegab_k)$ for each $k\in[N]$.
By boundedness of $\alpha_{\leq T}(\omegab)$, we have an upper bound $|Z_k| \leq T$ for any $k\in[N]$. The variance of $Z$ is bounded above as
\begin{equation*}
\sigma^2 := \E_{\omegab} |Z-\E_{\omegab} Z|^2 \leq \E_{\omegab} |Z|^2 \leq \E_{\omegab} [\alpha(\omegab)^2] = \|f\|^2_\rho. 
\end{equation*}
By Lemma A.2 and Theorem A.1 in \cite{lanthaler2023error}, it holds that, with probability at least $1-\delta$,
\begin{equation}
\label{term2}
\left|f^\sharp(\xb) - \E f^\sharp(\xb) \right| = \left|\frac{1}{N}\sum_{k=1}^N Z_k - \E_{\omegab} Z \right| \leq \frac{4T\log(2/\delta)}{N} + \sqrt{\frac{2\|f\|^2_\rho\log(2/\delta)}{N}}.
\end{equation}
Taking the square root for both sides of \eqref{term1}, and then adding it to \eqref{term2} gives 
\begin{equation*}
|f(\xb) - f^\sharp(\xb)| \leq \frac{\Big(\E_{\omegab}[\alpha(\omegab)^2]\Big)}{T} + \frac{4T\log(2/\delta)}{N} + \sqrt{\frac{2\|f\|^2_\rho\log(2/\delta)}{N}}.
\end{equation*}
Selecting $T=\sqrt{N}\|f\|_\rho$ gives the desired result. 
\end{proof}
The last theorem in this section presents the convergence rate of our proposed random feature model.
\begin{theorem}
\label{Thm:main}
Suppose that the conditions in Theorem \ref{Thm:kernel} and \ref{Thm:RF} hold. Then for any $\delta\in(0,1)$ there exists $c^\sharp_1,\dots,c^\sharp_N$ so that the function
\begin{equation*}
u^\sharp(\xb) = \sum_{k=1}^N c^\sharp_k\phi(\xb,\omegab_k) 
\end{equation*}
satisfies
\begin{equation*}
\| u - u^\sharp\|_{L^2(\Omega)} \leq Ch^\gamma\|u\|_\Hc + \frac{12\|u\|_\Hc\log(2/\delta)\vol(\Omega)}{\sqrt{N}}
\end{equation*}
with probability at least $1-\delta$ over $\omegab_1,\dots,\omegab_N$ drawn i.i.d from $\rho(\omegab)$.
\end{theorem}
\begin{proof}
Using the triangle inequality, we decompose the error as
\begin{equation*}
\| u - u^\sharp\|_{L^2(\Omega)} \leq \| u - \hat{u}\|_{L^2(\Omega)} + \| \hat{u} - u^\sharp\|_{L^2(\Omega)},  
\end{equation*}
where $\hat{u}\in\Hc$ is a minimizer of \eqref{Minnorm_KernelPDE} with interior collocation points $X_\Omega$ and collocation points on the boundary $X_{\partial\Omega}$. 
We directly apply Theorem \ref{Thm:kernel} to bound $\| u - \hat{u}\|_{L^2(\Omega)}$. The second term is bounded as
\begin{equation*}
\| \hat{u} - u^\sharp\|_{L^2(\Omega)} = \sqrt{ \int_{\Omega} \left| \hat{u}(\xb) - u^\sharp(\xb)\right|^2 d\xb} \leq \frac{12\|\hat{u}\|_\Hc\log(2/\delta)\vol(\Omega)}{\sqrt{N}} \leq \frac{12\|u\|_\Hc\log(2/\delta)\vol(\Omega)}{\sqrt{N}},
\end{equation*}
where the first inequality relies on the entry-wise bound in Theorem \ref{Thm:RF} and the second inequality holds since the strong solution $u\in\Hc$ satisfies the boundary conditions, and hence the minimizer $\hat{u}$ must satisfy $\|\hat{u}\|_\Hc \leq \|u\|_\Hc$. 
Adding the bounds together leads to the desired error bound.
\end{proof}

\section{Numerical experiments}
\label{Sec:numerics}
In this section, we test the performance of our proposed random feature method using several PDE benchmarks in the literature \cite{CHEN20211kernelPDE,WANG2024Extreme,xu2024kernel}. 
In addition, we numerically verify the convergence rate obtained in Theorem \ref{Thm:main}.
In all numerical experiments, we randomly generate training samples over the domains to train the models, and generate different test points to evaluate the PDE solutions. We compare the true solution and predicted solution on these test points (of size $M$) to compute the test error, which is defined as
\begin{equation*}
\mathop{\mathrm{Error}} = \frac{1}{M}\sum_{j=1}^{M} \Big( u(\xb_j) - \hat{u}(\xb_j) \Big)^2.
\end{equation*}
We consider nonlinear elliptic PDEs, nonlinear poisson PDEs, Allen-Cahn equation, and advection diffusion equation in the following sections. 
Detailed problem settings are clearly stated.
For each problem, we decide to use Gaussian random features. To implement Gaussian random features model, we should select an appropriate variance parameter $\sigma$. In theory, the variance parameter $\sigma$ affects the smoothness of functions in the corresponding RKHS $\Hc$. Recall the Gaussian kernel and the corresponding Gaussian density in Table \ref{tab:kernel} \footnote{The variance parameter $\sigma$ is indeed $\gamma$ in Table \ref{tab:kernel}. It is the variance parameter of the Gaussian density as well as the scale parameter of the Gaussian kernel.}, the larger variance parameter of the Gaussian density leads to a larger scale parameter in the Gaussian kernel, which results in a RKHS with functions that are more oscillatory. Conversely, a smaller scale parameter results in a more smooth RKHS. In practice, we could select the variance parameter $\sigma$ to match the regularity assumptions on the PDE solution. We could also treat it as a hyper-parameter and use cross-validation to select the best hyper-parameter.

We compared our proposed random features based method with PINN and ELM.
All experiments are implemented in Python based on Torch library. Notice that the training of ELM can be implemented by using a non-linear least-squares method or nonlinear iterative schemes such as Picard or Newton iterations, see examples \cite{DONG20211ELM, SHANG2023RPG, Shang2024RPG, SUN2025LRDG} among many others. We also test this training strategy in all numerical experiments.
Our codes are available on the repository: \url{https://github.com/liaochunyang/RF_PDE}.

\subsection{Nonlinear Elliptic PDEs}
We test with the instance of nonlinear elliptic PDE
\begin{equation*}
\begin{aligned}
-\Delta u(\xb) + u(\xb)^3 &= f(\xb), \quad && \xb\in\Omega \\
u(\xb) &= g(\xb), \quad && \xb\in\partial\Omega, 
\end{aligned}
\end{equation*}
where $\Omega = [0,1]^2$. The solution $u(\xb)=\sin(\pi x_1)\sin(\pi x_2) + 4\sin(4\pi x_1)\sin(4\pi x_2)$, and the right-hand side $f(\xb)$ is computed accordingly via the solution $u(\xb)$. The boundary condition is $g(\xb)=0$. This example was also considered in \cite{CHEN20211kernelPDE, xu2024kernel}. 
We randomly sample 1000 random features ($N=1000$) from normal distribution $\Nc(0,\sigma^2\Ib)$ with variance $\sigma^2=100$. 
The random feature model is trained on $M_\Omega=900$ interior collocation points and 124 collocation points on the boundary $\partial\Omega$. We take the uniform grids of size $100\times100$ as our test points. 
In Figure \ref{Fig:Nonlinear_PDE}, we show predicted solution using our proposed random feature model, true solution, and entry-wise absolute errors at test points. We observe that our proposed random feature model provides an accurate prediction of the true solution.
\begin{figure}[!htbp]
\centering 
\includegraphics[width=160mm]{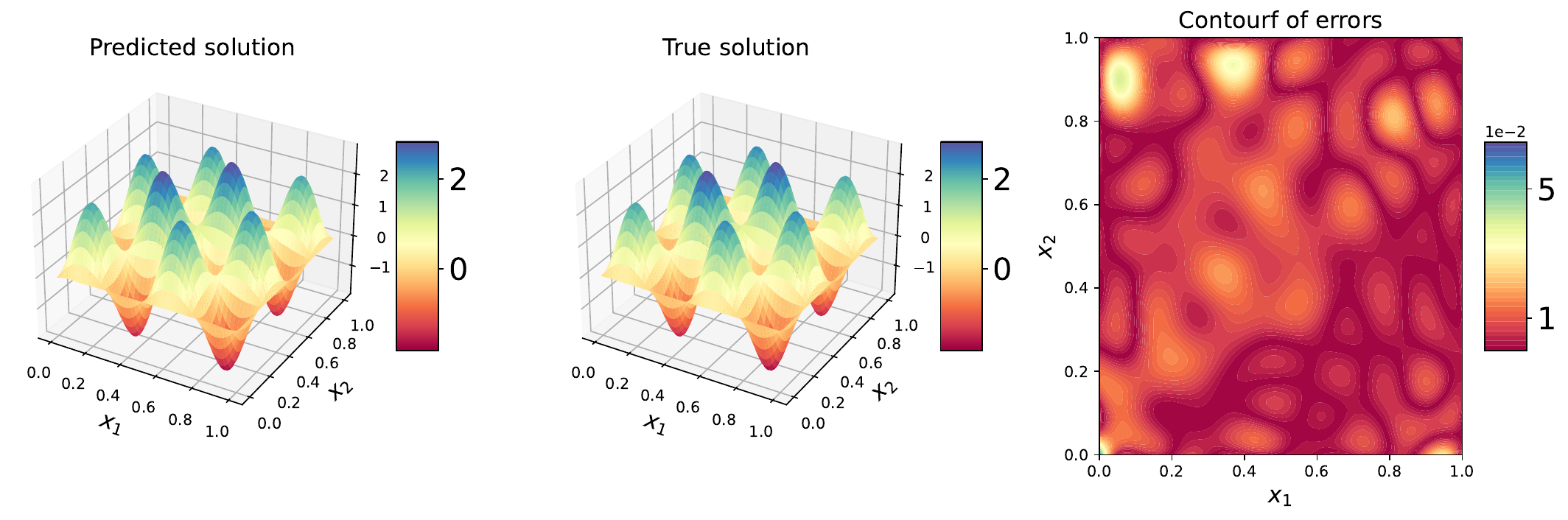}
\caption{Numerical results of the nonlinear elliptic PDE: we show the predicted solution, the true solution, and the entry-wise absolute error.}
\label{Fig:Nonlinear_PDE}
\end{figure}

We also compare the training epochs, test errors and training times of our proposed method with PINN and ELM, see results in Table \ref{tab:Nonlinear_PDE}. 
The PINN model has 2 hidden layers and each layer has 64 neurons. The nonlinear activation function is tanh function. 
The ELM model has 1 hidden layers with $N=1000$ neurons (the same number of neurons as the random feature model). Random weights in the ELM model are uniformly sampled from $\Uc[-R,R]$ where $R=0.05$. There is no bias term in the hidden layer. As the reference suggested \cite{WANG2024Extreme}, parameter $R$ should be selected using cross-validation. Here, we reported the lowest test error and the corresponding epochs and training time.
PINN and ELM models are trained on the same training samples and are tested on the uniform grid as well.
The first four rows of Table \ref{tab:Nonlinear_PDE} contains the numerical results of models trained by a stochastic gradient descent (SGD)-type algorithm. Then, we train our proposed RF model and ELM model using nonlinear least-squares (NLS) method and report results.
We observe that training PINN is complicated, which requires more epochs, and hence longer training time. If we set the same number of epochs, then the PINN model gives a bad prediction compared with random feature method. Moreover, our proposed method has smaller test error.
The ELM model and the RF model have similar training times, but our RF model has much lower test error than the ELM model.
Moreover, random features model and ELM can be trained efficiently by using nonlinear least-squares method. Under this training scheme, our proposed method also has better performance.
Overall, our proposed random feature model outperforms PINN and ELM in this example.

\begin{table}[!htbp]
\centering
\begin{tabular}{|c|c|c|c|}
\hline
 Method & Epochs & Test error & Training Time (Seconds)  \\\hline
 RF & 1000 & $6.71 \times 10^{-5}$ & 80.23 \\\hline
 PINN & 1000 & 1.23 & 13.42 \\\hline
 PINN & 10000 & $1.35 \times 10^{-2}$ &  169.28 \\\hline
 ELM & 1000 & 1.22 & 68.50 \\\hline
 RF-NLS & - & $2.59 \times 10^{-6}$ & 6.46\\\hline
 ELM-NLS & - & $5.39 \times 10^{-3}$ & 9.02\\\hline
\end{tabular}
\caption{Numerical results of the nonlinear elliptic PDE: we compare our proposed random feature method with PINN and ELM.}
\label{tab:Nonlinear_PDE}
\end{table}

Finally, we numerically verify the convergence rate of nonlinear elliptic PDE. We first fix the number of random features to be $N=100$ and varies the number of collocation points. We sample $M_\Omega = 400, 900, 1600$ points uniformly in the domain, and  $M_{\partial\Omega}=84, 124, 164$ points uniformly on the boundary. We sample another set of 100 test points and evaluate the test errors. 
In the second experiment, we fix $M_\Omega = 400$ interior collocation points and $M_{\partial\Omega}=84$ points on the boundary. We take different numbers of random features, i.e. $N=100,200,300$. We sample another set of 100 test points and evaluate the test errors. In Figure \ref{Fig:Nonlinear_PDE_rate}, we show the test errors as a function of the number of collocation points and a function of the number of random features, respectively. 
For each point in the figure, we repeat the experiments 10 times and take the average.
\begin{figure}[!htbp]
\centering 
\subfigure[]{\includegraphics[width=60mm]{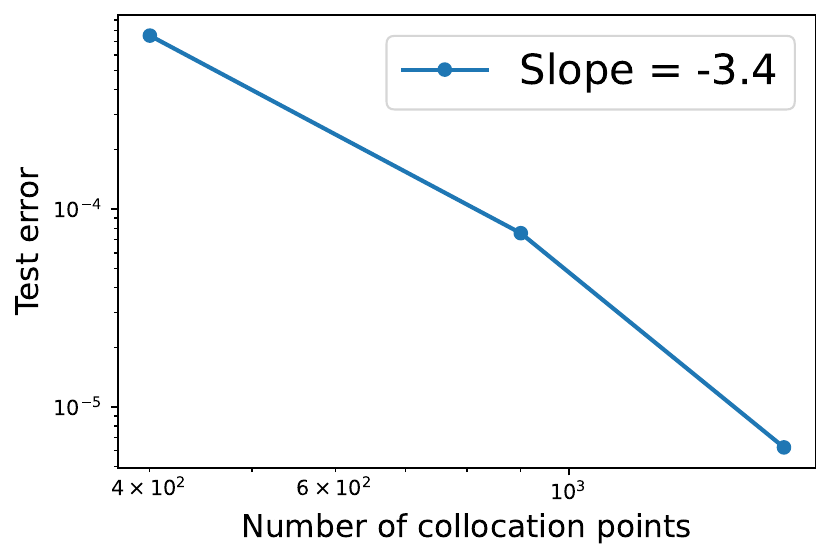}}
\subfigure[]{\includegraphics[width=60mm]{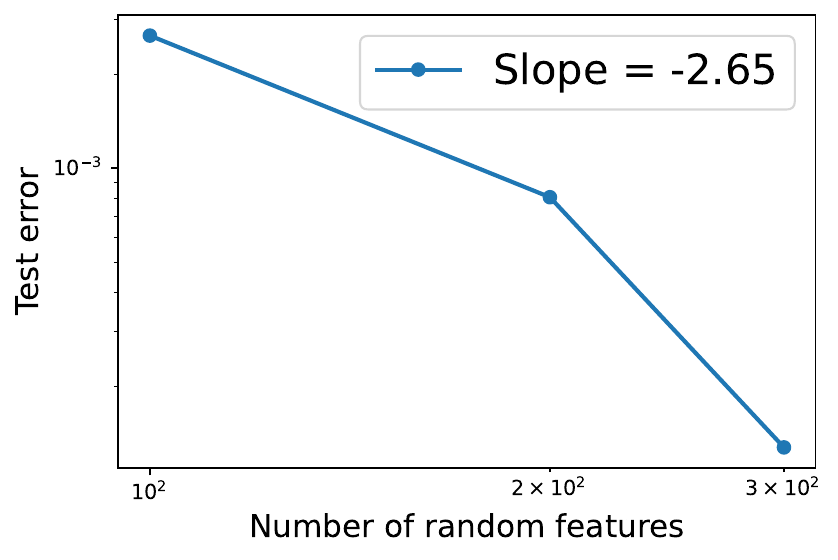}}
\caption{Test error as a function of the number of collocation points, and the number of random features, respectively. Slopes reported in the legends denote empirical convergence rates.}
\label{Fig:Nonlinear_PDE_rate}
\end{figure}

\subsection{Nonlinear Poisson PDE}
In this section, we test our proposed method with high-dimensional nonlinear Poisson PDEs. We consider the domain $\Omega = [-1,1]^d$ and the following problem defined on $\Omega$,
\begin{equation*}
\begin{aligned}
-\nabla\cdot (a(u)\nabla u) &= f(\xb), \quad && \xb\in\Omega \\
u(\xb) &= g(\xb), \quad && \xb\in\partial\Omega,
\end{aligned}
\end{equation*}
where $a(u)=u^2-u$. The solution is crafted as $u(\xb) = \exp(-\frac{1}{d}\sum_{i=1}^dx_i)$. The function $g(\xb)=u(\xb)$ on the boundary, and the right-hand side $f(\xb)$ is computed using the true solution, which has the explicit expression
\begin{equation*}
f(\xb) = \frac{1}{d}\left[-3\exp\left(-\frac{3}{d}\sum_{i=1}^dx_i\right) + 2\exp\left(-\frac{2}{d}\sum_{i=1}^dx_i\right)\right].  
\end{equation*}
We first test with the 2D nonliear Poisson PDE and show the predicted solution, true solution, and entry-wise absolute errors in Figure \ref{Fig:Poisson}. 
The training samples of size 1024 contains 900 interior collocation points and 124 boundary points, which are uniformly generated over the domain $\Omega = [-1,1]^2$ and its boundary $\partial \Omega$, respectively. We randomly generate 500 test samples to evaluate the performance. 

\begin{figure}[!htbp]
\centering 
\includegraphics[width=160mm]{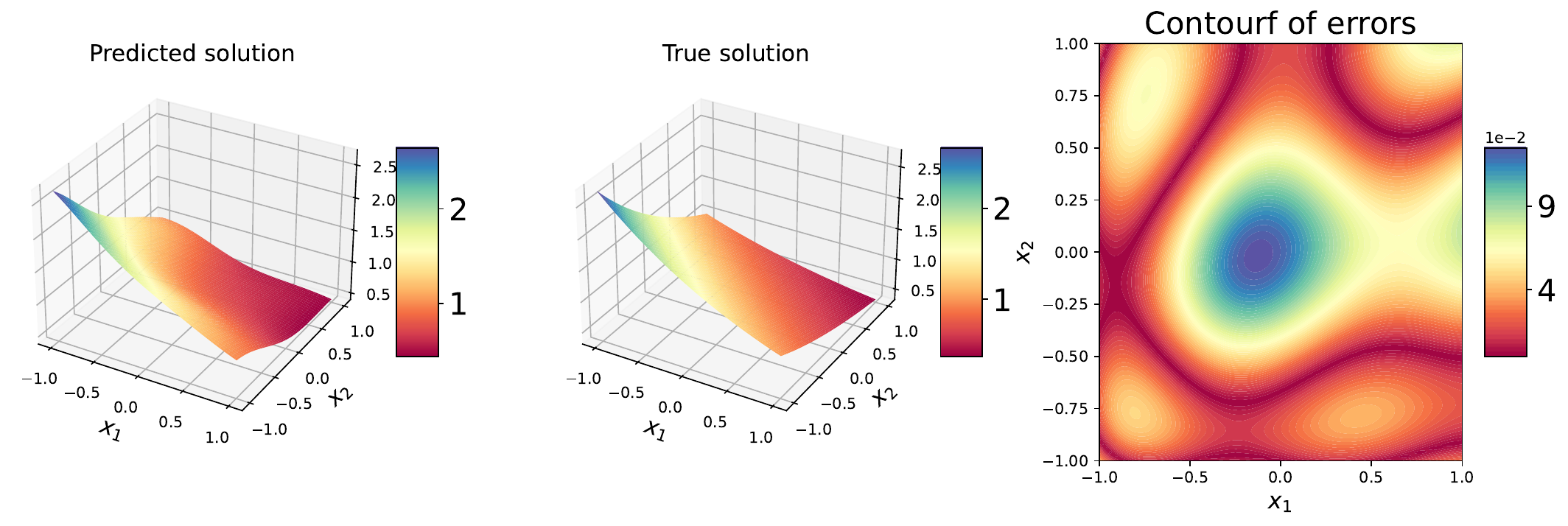}
\caption{Numerical results of nonlinear Poisson PDE: we show the predicted solution using random feature model, true solution, and corresponding entry-wise errors at test points.}
\label{Fig:Poisson}
\end{figure}

We first note that the accuracy of our model is related to the choice of variance of Gaussian random feature, but it is not sensitive to the variance. Usually, we can tune the hyperparameter using cross-validation. In Table \ref{tab:Poisson_var}, we report the variances and the corresponding test errors for $d=8$ dimension nonlinear Poisson PDE. In this example, it is better to use small variance, but the test error is not very sensitive to the choice of variance when it is smaller than some threshold.

\begin{table}[!htbp]
\centering
\begin{tabular}{|c|c|c|c|c|c|}
\hline
$\sigma^2$ & 100 & 1 & 0.04 & 0.01 & 0.0025   \\\hline
Test error & $2.14\times10^{-1}$ & $3.62\times10^{-2}$ & $7.31\times10^{-3}$ & $4.81\times10^{-4}$ & $4.05\times10^{-4}$ \\\hline
\end{tabular}
\caption{Variances of Gaussian random features and the corresponding test errors for nonlinear Poisson PDE ($d=8$). The training sample size $M=1024$ and each error is calculated on a set of test samples with size 100. Number of random features $N$ is set to 1000, following the standard Normal distribution.}
\label{tab:Poisson_var}
\end{table}

Furthermore, we compare our proposed RF-based model with PINN model and ELM model. The physics-informed neural network has 2 hidden layers and 64 neurons at each layer. The ELM model has 1 hidden layer with number of neurons reported in Table \ref{tab:Poisson}. In all the following experiments, random weights in ELM model are randomly sampled from uniform distribution $\Uc[-0.05,0.05]$ and there is no bias term.
We select tanh function as the activation function for both PINN and ELM.
We test 2D nonlinear Poisson PDE as well as high-dimensional nonlinear Poisson PDEs ($d=2,4$ and $8$). 
We also train random feature model and ELM using nonlinear least-squares method.
All numerical results are summarized in Table \ref{tab:Poisson}. 
We observe that our proposed random feature model achieves similar performance or even beats the PINN models in terms of the test error. 
While ELM models have less training times, our proposed random feature based models always outperform the ELM models cross all examples in terms of the test error.

\begin{table}[!htbp]
\centering
\begin{tabular}{|c|c|c|c|c|c|}
\hline
Dimension & Method & $N$ & Epochs & Test error & Training Time (Seconds)  \\\hline
\multirow{5}{*}{$d=2$} & RF & 500 & 1000 & $3.19 \times 10^{-3}$ & 42.29 \\\cline{2-6}
 & PINN & - & 2000 & $5.28 \times 10^{-3}$ &  51.46 \\\cline{2-6}
  & ELM & 500 & 2000 & $3.21 \times 10^{-2}$ &  28.25 \\\cline{2-6} 
  & RF-NLS & 500 & - & {\bf $6.39 \times 10^{-5}$} & 6.37 \\\cline{2-6}
 & ELM-NLS & 500 & - & $1.22 \times 10^{-4}$ & 5.33 \\\hline
\multirow{5}{*}{$d=4$} & RF & 500 & 1000 & $6.51 \times 10^{-3}$ & 51.24 \\\cline{2-6}
 & PINN & - & 2000 & $4.27 \times 10^{-3}$ &  55.49 \\\cline{2-6}
 & ELM & 500 & 1000 & $3.00 \times 10^{-2}$ & 53.04 \\\cline{2-6}
 & RF-NLS & 500 & - & {\bf $1.92 \times 10^{-4}$} & 7.11 \\\cline{2-6}
 & ELM-NLS & 500 & - & $1.12 \times 10^{-3}$ & 5.88 \\\hline
 \multirow{5}{*}{$d=8$} & RF & 100 & 1500 & $8.43\times 10^{-3}$ & 37.84 \\\cline{2-6}
 & PINN & - & 2000 & $9.91 \times 10^{-3}$ &  54.14 \\\cline{2-6}
 & ELM & 100 & 1000 & $2.93 \times 10^{-2}$ & 33.05 \\\cline{2-6}
 & RF-NLS & 500 & - & {\bf $7.35 \times 10^{-4}$} & 7.76 \\\cline{2-6}
 & ELM-NLS & 500 & - & $1.25 \times 10^{-3}$ & 4.82 \\\hline
\end{tabular}
\caption{Comparison between random feature, PINN and ELM models for the high-dimensional nonlinear Poisson PDEs. We report the number of random features, number of epochs, test error, and training time for each model.}
\label{tab:Poisson}
\end{table}

In Figure \ref{Fig:Poisson_rate}, we report the empirical convergence rates for the nonlinear Poisson equations. Figure \ref{Fig:Poisson_rate}(a) shows the test error as a function of the number of collocation points. For each dimension, we fix $N=100$ random features and varies the number of collocation points. We uniformly sample $M_\Omega = 100,400,900$ interior points, and $M_{\partial\Omega} = 44,84,124$ boundary points. We sample a different set of 100 test points to evaluate the test errors. Figure \ref{Fig:Poisson_rate}(b) shows the test error as a function of the number of random features. For each dimension, we fix the collocation points of size 484 with 400 interior collocation points. We take $N=100,400$ random features. We average over 10 experiments to produce each point in Figure \ref{Fig:Poisson_rate}. We observe that our theoretical upper bound is numerically verified, which indicates that our proposed random feature model does not suffer "curse of dimensionality". 

\begin{figure}[!htbp]
\centering 
\subfigure[]{\includegraphics[width=60mm]{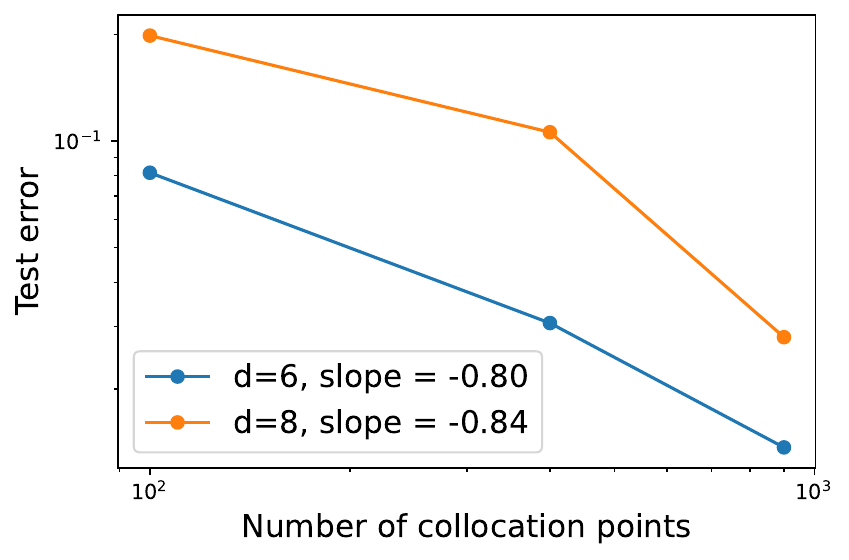}}
\subfigure[]{\includegraphics[width=60mm]{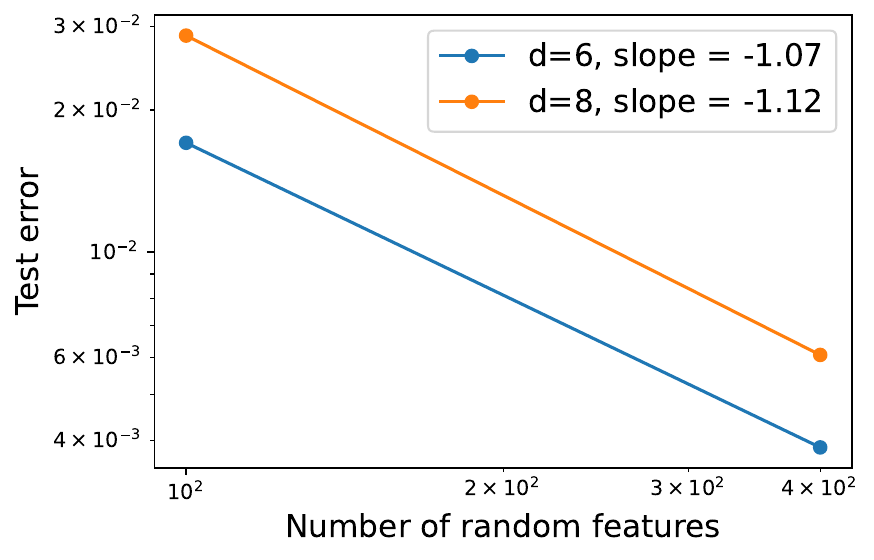}}
\caption{Test error as a function of the number of collocation points, and the number of random
features, respectively. Slopes reported in the legends denote empirical convergence rates.}
\label{Fig:Poisson_rate}
\end{figure}

\subsection{Allen-Cahn Equation}
Next, we consider a 2D stationary Allen-Cahn equation with a source function and Dirichlet boundary conditions, i.e.
\begin{equation*}
\Delta u + \gamma(u^m-u) = f(\xb), \quad \xb\in[0,1]^2,
\end{equation*}
where $\gamma=1$ and $m=3$. The solution takes the form $u(\xb)=\sin(2\pi ax_1)\cos(2\pi ax_2)$, and the function $f(\xb)$ is computed using the solution $u(\xb)$. Positive parameter $a$ controls the frequency of the solutions. We test two cases $a=1$ and $a=10$.

\begin{table}[!htbp]
\centering
\begin{tabular}{|c|c|c|c|c|c|c|}
\hline
Frequency & Method & $N$ & $\sigma^2$ & Epochs & Test error & Training Time (Seconds)  \\\hline
\multirow{5}{*}{$a=1$} & RF & 200 & $10^2$ & 1000 & $7.80 \times 10^{-6}$ & 14.27 \\\cline{2-7}
 & PINN & - & - & 1000 & $2.45 \times 10^{-1}$ & 8.64 \\\cline{2-7}
 & ELM & 200 & - & 1000 & $2.45 \times 10^{-1}$ & 17.40 \\\cline{2-7}
 & RF-NLS & 200 & $10^2$ & - & $3.29 \times 10^{-8}$ & 0.63 \\\cline{2-7}
 & ELM-NLS & 200 & - & - & $6.16 \times 10^{-3}$ & 0.64 \\\hline
\multirow{5}{*}{$a=10$} & RF & 200 & $100^2$ & 2000 & $1.12 \times 10^{-4}$ & 25.37 \\\cline{2-7}
 & PINN & - & - & 2000 & $1.04 \times 10^{+1}$ &  18.48 \\\cline{2-7}
 & ELM & 400 & - & 1000 & $2.45 \times 10^{-1}$ & 18.27 \\\cline{2-7}
 & RF-NLS & 200 & $100^2$ & - & $2.10 \times 10^{-4}$ & 1.63 \\\cline{2-7}
 & ELM-NLS & 200 & - & - & $3.16 \times 10^{-3}$ & 1.34 \\\hline
\end{tabular}
\caption{Comparison between our proposed random feature model and two state-of-the-art benchmarks (PINN and ELM) for the Allen-Cahn equations. We report number of epochs, test error, and training time for each model. For the random feature model, we also report the number of random features and variance of Gaussian random features. }
\label{tab:Allen_Cahn}
\end{table}

We first compare the test performances. In each case, we randomly sample $M=1024$ points with $M_\Omega=900$ interior collocation points to train models. 
The test performances for both models are evaluated on 100 test samples, which are uniformly generated over the domain and boundary. We report the parameter selections and summarize the numerical results in Table \ref{tab:Allen_Cahn}.
From the numerical results, we observe that the random feature models beat PINNs and ELMS cross all cases, especially when the frequency parameter $a$ is large.  
It might be related to the known "spectral bias" of neural networks \cite{basri2020spectral,rahaman2019spectral,basri2019spectral}. Precisely, neural networks trained by gradient descent fit a low frequency function before a high frequency one. Therefore, it is difficult for PINNs to learn the high frequency PDE solutions.
To alleviate "spectral bias" and learn high frequency functions, previous work proposed to use Fourier features and both theoretically and empirically showed that a Fourier feature mapping can improve the performance \cite{matthew2020fourier}. Fourier features have been used to solve high frequency PDEs, see \cite{WANG2021fourier}. 
Our theory suggests that a larger variance parameter $\sigma$ results in a more oscillatory RKHS. Our numerical experiemnts also verify the theory since higher frequency in the PDE solutions leads to a larger variance of Gaussian random feature.
Moreover, it requires more random features and epochs to train the random feature model as the frequency increasing. In Figure \ref{Fig:Allen-cahn}, we show predicted solution, true solution and entry-wise errors at test points.

\begin{figure}[!htbp]
\centering 
\includegraphics[width=160mm]{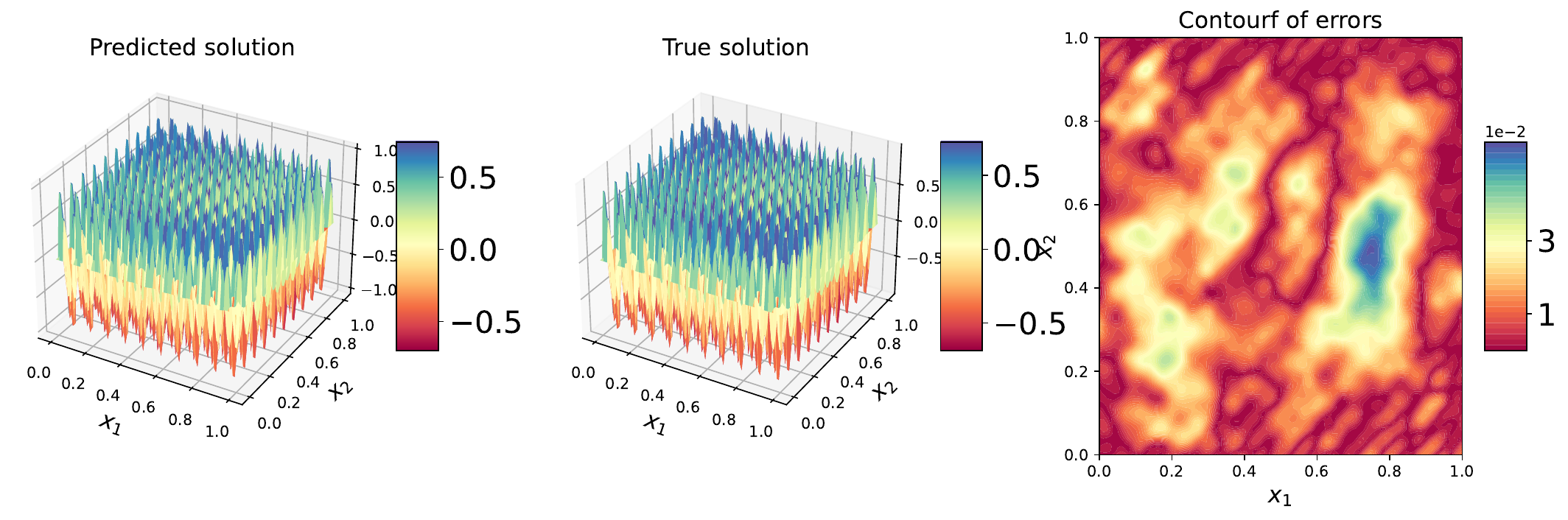}
\caption{Numerical results for the Allen-Cahn equation with frequency parameter $a=10$: we show the predicted solution, true solution, and corresponding entry-wise errors at test points.}
\label{Fig:Allen-cahn}
\end{figure}

Figure \ref{Fig:Allen_rate} illustrates the numerical verifications of the convergence rate of Allen-Cahn equation.  
We first show the test error as a function of the number of collocation points. In this experiment, we fix the number of random features to be $N=100$. To produce the collocation points, we uniformly sample $M_\Omega = 400, 900, 1600$ points in the domain, and $M_{\partial\Omega}=84, 124, 164$ points on the boundary. We sample another set of 100 test points to evaluate the test errors. 
In the second experiment, where we show the test error as a function of the number of random features, we uniformly sample and then fix $M_\Omega = 400$ interior collocation points and $M_{\partial\Omega}=124$ points on the boundary. We sample another set of 100 test points and evaluate the test errors.
We generate $N=100,200,400$ random features from Gaussian distribution. In Figure \ref{Fig:Allen_rate}, we show the test errors as a function of the number of collocation points and a function of the number of random features, respectively. 
For each point in the figure, we repeat the experiments 10 times and take the average.

\begin{figure}[!htbp]
\centering 
\subfigure[]{\includegraphics[width=60mm]{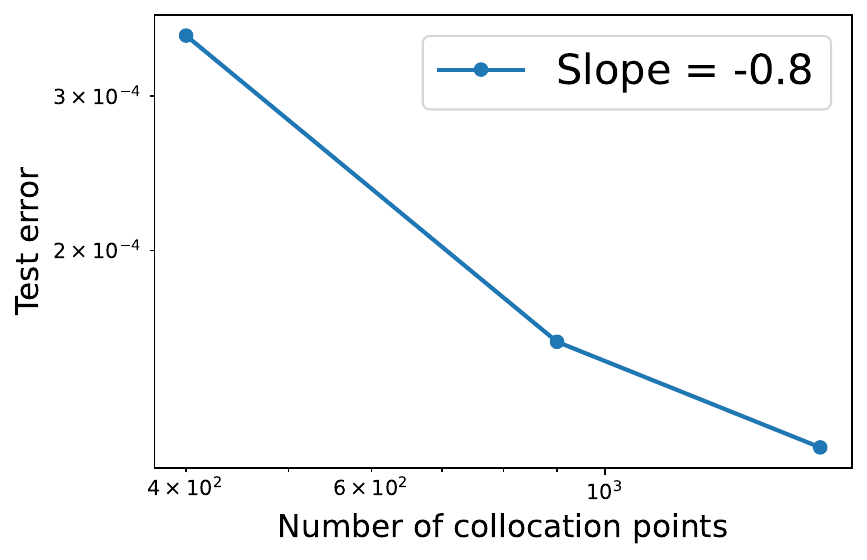}}
\subfigure[]{\includegraphics[width=60mm]{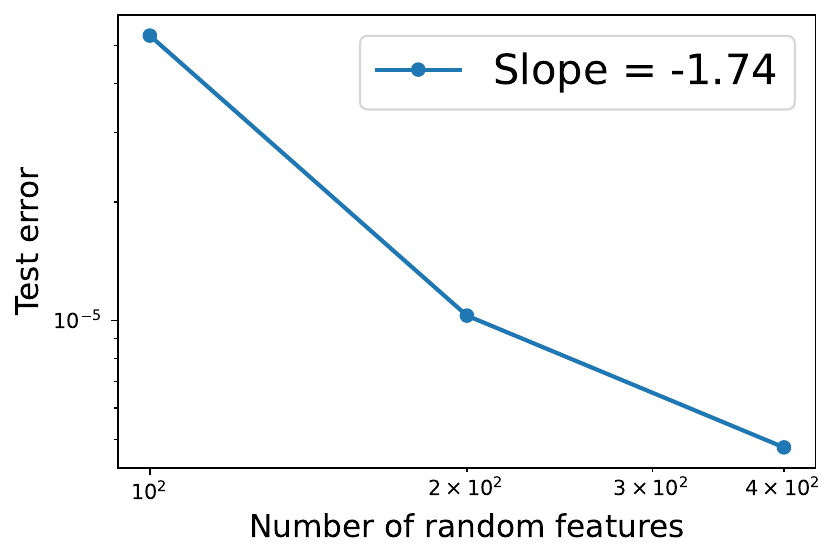}}
\caption{Allen-Cahn equation: test error as a function of the number of collocation points, and the number of random
features, respectively. Slopes reported in the legends denote empirical convergence rates.}
\label{Fig:Allen_rate}
\end{figure}

\subsection{Advection Diffusion Equation}
Finally, we test our method with advection diffusion equation. Consider the initial boundary value problem on the spatial-temporal domain $(x,t)\in[-1,1]\times[0,1]$, the PDE is 
\begin{equation*}
\begin{aligned}
u_t - u_{xx} + u_x &= f(x,t), \quad && (x,t) \in [-1,1]\times[0,1] \\
u(x,t) &= g(x,t), \quad && (x,t)\in\{-1,1\}\times[0,1], \\
u(x,0) &= h(x), \quad && x\in[-1,1]. 
\end{aligned}
\end{equation*}
The true solution is employed as $u(x,t) = \sin(x)\exp(-t)$. Functions $f(x,t)$, $g(x,t)$, and $h(x)$ are set according to the true solution. 
When we simulate this problem, we treat the time variable $t$ in the same way as the spatial variable $x$. We uniformly generate 1000 training samples over $[-1,1]\times[0,1]$. We enforce the boundary condition on 100 collocations points and the initial condition on 200 collocations points. 
Gaussian random features ($N=100$) are randomly sampled from standard Normal distribution (variance $\sigma^2=1$). 
The PINN model has 2 hidden layers with 64 neurons at each layer. 
The ELM model has 1 hidden layer with 100 neurons. Random weights are randomly sampled from uniform distribution $\Nc[-0.05,0.05]$. 
Since this example is indeed a linear PDE, we can train random feature model and ELM by solving a linear system. We consider the least-squares solution (LS) if the linear system is overdetermined and the min-norm interpolation solution when the linear system is underdetermined.
We report number of epochs, the test errors and training times for both models in Table \ref{tab:advection}. 
In this example, our proposed method achieves similar test error as PINN and outperforms ELM provided that they have the same amount of trainable parameters. In terms of training process, training RF model is simpler in the sense that it requires less epochs and the training time of random feature model is around 50\% of that of PINN model. Moreover, solving least-squares problem is much faster than implementing a SGD-type algorithm at the expense of losing test accuracy.
In Figure \ref{Fig:Advection}, we compare the predicted solution and true solution at various times $t = 0.1, 0.5, 0.9$ to further highlight the ability of our method in learning the true solution.

\begin{table}[!htbp]
\centering
\begin{tabular}{|c|c|c|c|}
\hline
 Method & Epochs & Test error & Training Time (Seconds)  \\\hline
 RF & 600 & $4.82\times 10^{-4}$ & 14.28 \\\hline
 PINN & 1000 & $4.64\times 10^{-4}$ & 30.78 \\\hline
 ELM & 1000 & $9.76 \times 10^{-3}$ & 23.48 \\\hline
 RF-LS & - & $3.38 \times 10^{-2}$ & 0.21 \\\hline
 ELM-LS & - & $3.53 \times 10^{-2}$ & 0.23 \\\hline
\end{tabular}
\caption{Numerical results of the advection diffusion equation: we compare our proposed random feature method with PINN and ELM.}
\label{tab:advection}
\end{table}

\begin{figure}[!htbp]
\centering 
\subfigure[$t=0.1$]{\includegraphics[width=45mm]{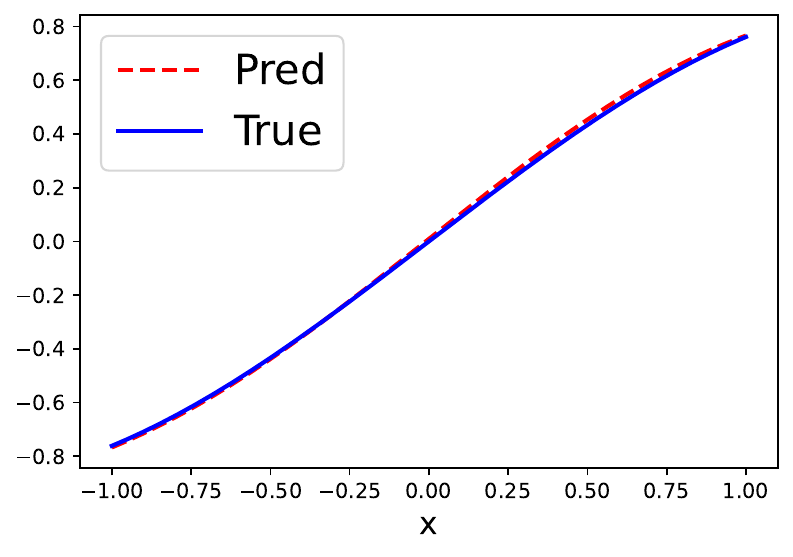}}
\subfigure[$t=0.5$]{\includegraphics[width=45mm]{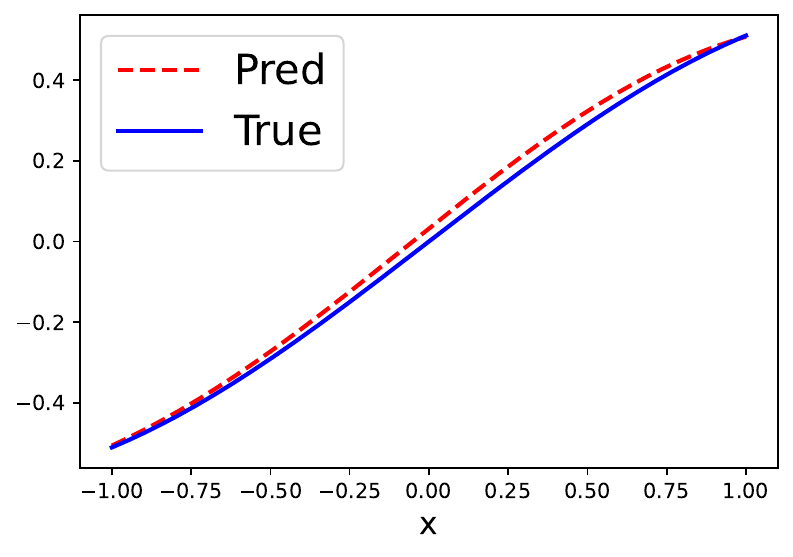}}
\subfigure[$t=0.9$]{\includegraphics[width=45mm]{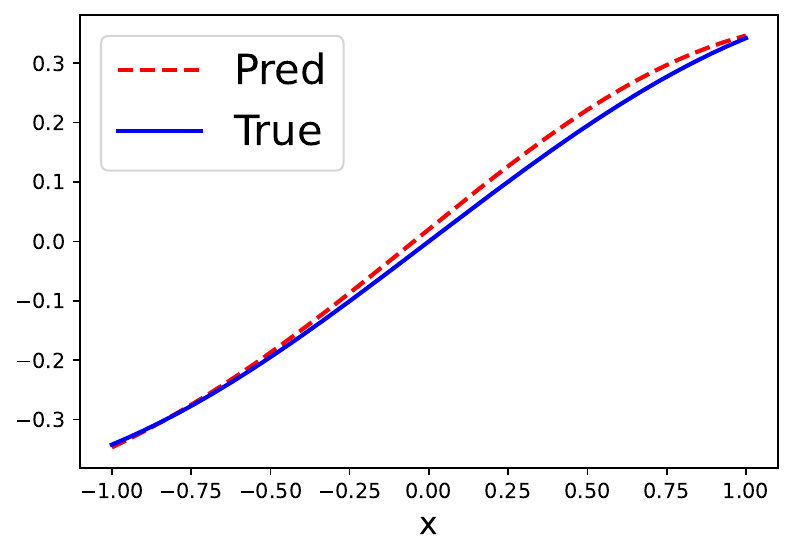}}
\caption{Numerical results of advection diffusion equation: predicted solution and true solution at time slices $t = 0.1, 0.5, 0.9$.}
\label{Fig:Advection}
\end{figure}

Finally, we perform a convergence study for advection-diffusion equation. In Figure \ref{Fig:Advection_rate}(a), we show the test errors as a function of the number of collocation points. We use $M=100,200,400$ collocation points, which are uniformly generated from $[-1,1]\times[0,1]$. We use 100 points on the boundary and 200 points for initial values. In Figure \ref{Fig:Advection_rate}(b), we use 200 interior points, 100 boundary points, and 200 points for initial condition, respectively. We vary the number of random features, i.e $N=100,200,400$. For each point in the figure, we take the average over 10 repetitions of the experiment. 

\begin{figure}[!htbp]
\centering 
\subfigure[]{\includegraphics[width=60mm]{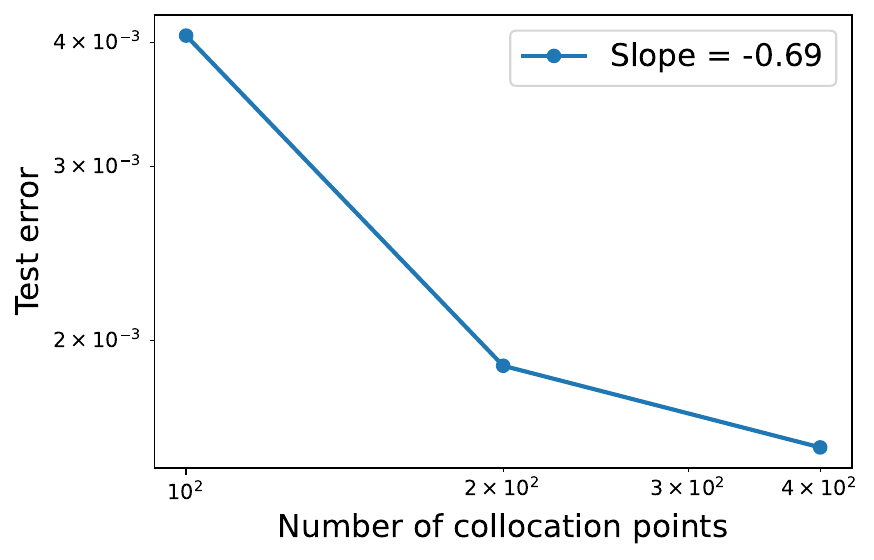}}
\subfigure[]{\includegraphics[width=60mm]{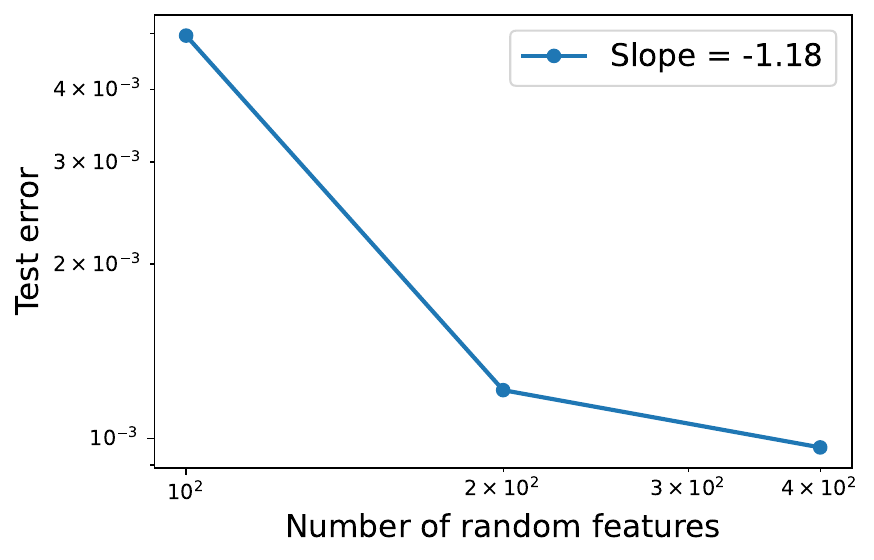}}
\caption{Advection-diffusion equation: test error as a function of the number of collocation points, and the number of random features, respectively. Slopes reported in the legends denote empirical convergence rates.}
\label{Fig:Advection_rate}
\end{figure}

\section{Conclusion}
\label{Sec:conclusion}
In this paper, we propose a random feature model for solving partial differential equations along with an error analysis. By utilizing some techniques from probability, we provide convergence rates of our proposed method under some mild assumptions on the PDE. Our framework allows convenient implementation and efficient computation. Moreover, it easily scales to massive collocation points, which are necessary for solving challenging PDEs. We test our method on several PDE benchmarks. The numerical experiments indicate that our method either matches or beats state-of-the-art models and reduces the computational cost. 

Finally, we conclude with some directions for future work. First, our analysis does not directly address the minimizer we obtained by solving an optimization problem. It requires us to analyze a min-norm minimization problem with some nonlinear constraints. 
Second, while it is natural to sample random features from the Fourier transform density, it is advantageous to sample from a different density which has been shown to yield better performance.
Third, we assume that the PDE at hand is well- defined pointwise and has a unique strong solution. Extension our framework to weak solution is left for future work. 

\bibliography{refs}

\end{document}

%% file: CL_marco.tex
\usepackage[utf8]{inputenc}
\usepackage[left=1in, right=1in,top=0.8in,bottom=1.5in]{geometry}
\usepackage{setspace}

\usepackage{amsfonts,amsmath,amssymb,amsthm}
\usepackage[shortlabels]{enumitem}
\usepackage[hidelinks]{hyperref}
\usepackage{graphics,graphicx,subfigure}
\usepackage{caption}
\usepackage{romannum}
\usepackage{color}
\usepackage{tikz}
\usepackage{multirow}
\usepackage{bbm}
\usepackage{romannum}

\usepackage{algorithm, xcolor}
\usepackage{algpseudocode}

\theoremstyle{plain}
\theoremstyle{definition}
\newtheorem{lemma}{Lemma}

\newtheorem{theorem}[lemma]{Theorem}

\newtheorem{assumption}{Assumption}

\newtheorem*{remark}{Remark}

\newtheorem*{example}{Example}
\theoremstyle{remark}

\newcommand{\R}{\mathbb{R}}
\newcommand{\C}{\mathbb{C}}
\newcommand{\E}{\mathbb{E}}
\newcommand{\N}{\mathbb{N}}
\newcommand{\Pbb}{\mathbb{P}}

\newcommand{\Ab}{\mathbf{A}}
\newcommand{\Bb}{\mathbf{B}}
\newcommand{\Ib}{\mathbf{I}}

\newcommand{\xb}{\mathbf{x}}
\newcommand{\yb}{\mathbf{y}}
\newcommand{\cb}{\mathbf{c}}

\newcommand{\fb}{\mathbf{f}}
\newcommand{\gb}{\mathbf{g}}

\newcommand{\Oc}{\mathcal{O}}
\newcommand{\Uc}{\mathcal{U}}

\newcommand{\Nc}{\mathcal{N}}
\newcommand{\Fc}{\mathcal{F}}

\newcommand{\Pc}{\mathcal{P}}
\newcommand{\Bc}{\mathcal{B}}
\newcommand{\Hc}{\mathcal{H}}

\newcommand{\argmin}{\mathop{\mathrm{argmin}}}
\newcommand{\minimize}{\mathop{\mathrm{minimize}}}
\newcommand{\vol}{\mathop{\mathrm{vol}}}
\newcommand{\indicator}{\mathbbm{1}}
\newcommand{\omegab}{\boldsymbol\omega}